\documentclass[12pt]{amsart}
\usepackage{amssymb}
\usepackage{color}

\makeatletter
\def\LaTeX{\leavevmode L\raise.42ex
    \hbox{\kern-.3em\size{\sf@size}{0pt}\selectfont A}\kern-.15em\TeX}
\makeatother

\newcommand{\BibTeX}{{\rm B\kern-.05em{\sc i\kern-.025emb}\kern-.08em\TeX}}

\newcommand{\R}{\mathbb{R}}
\newcommand{\HH}{\mathbb{H}}
\newcommand{\D}{\mathbb{D}}
\newcommand{\transpose}{\text{T}}

\newtheorem{theorem}{Theorem}[section]
\newtheorem{proposition}[theorem]{Proposition}

\newtheorem{corollary}[theorem]{Corollary}
 
\newtheorem{rem}[theorem]{Remark}

\theoremstyle{definition}

\theoremstyle{remark}



\makeatletter
\def\alphenumi{%
  \def\theenumi{\alph{enumi}}%
  \def\p@enumi{\theenumi}%
  \def\labelenumi{(\@alph\c@enumi)}}
\makeatother

\newcounter{bk}

\newcounter{gm}

\begin{document}

\title[The Euler and Navier-Stokes equations on the hyperbolic plane] 
{The Euler and Navier-Stokes equations on the hyperbolic plane} 

%\author{Boris Khesin and Gerard Misio\l ek}
%
%\thanks{
%B.K.: School of Mathematics, Institute for Advanced Study, Princeton, NJ 08540, USA and Department of %Mathematics,
%University of Toronto, Toronto, ON M5S 2E4, Canada;
%e-mail: {\tt khesin@math.toronto.edu}}
% 
%\thanks{G.M.: School of Mathematics, Institute for Advanced Study, Princeton, NJ 08540, USA and 
%Department of Mathematics,
%University of Notre Dame, IN 46556, USA;
%e-mail: \tt{gmisiole@nd.edu}} 

%\maketitle

\author{Boris Khesin}
\address{B.K.: School of Mathematics, Institute for Advanced Study, Princeton, NJ 08450, USA 
and Department of Mathematics, University of Toronto, Toronto, ON M5S 2E4, Canada}
%Research address for author one
%\curraddr{}
%Current address for author one
\email{khesin@math.toronto.edu}
%  \thanks will become a 1st page footnote.
%  Don't type a period at the end; it will be supplied.
\thanks{B.K. was partially supported by the Simonyi Fund and an NSERC research grant}
\author{Gerard Misio\l ek}
\address{G.M.: School of Mathematics, Institute for Advanced Study, Princeton, NJ 08450, USA and 
Department of Mathematics, University of Notre Dame, IN 46556, USA} 
\email{gmisiole@nd.edu} 
\thanks{G.M. was partially supported by the James D. Wolfensohn Fund and 
the Friends of the Institute for Advanced Study Fund}

\maketitle
%\keywords{} 
%Math Subject Classifications 
%\subjclass{} 
\date{\today}

\begin{abstract} 
We show that non-uniqueness of the  Leray-Hopf solutions of the Navier--Stokes equation 
on the hyperbolic plane $\mathbb{H}^2$ observed in \cite{cc} 
%observed in \cite{cc} 
is a consequence of the Hodge decomposition. 
%of 1-forms (or vector fields). 
We show that this phenomenon does not occur on $\HH^n$ whenever $n\ge 3$.
We also describe the corresponding general Hamiltonian setting of hydrodynamics 
on complete Riemannian manifolds, which includes the hyperbolic setting. 
\end{abstract} 

\maketitle

%\numberwithin{equation}{section} 

%%%%%%%%%%%%%%%%%%%%%%%%%%%%%%%%%%%%%%%%%%%%%%%%%%%%
%%%%%%%%%%%%%%%%%%%%%%%%%%%%%%%%%%%%%%%%%%%%%%%%%%

\section*{Introduction} 

Consider the initial value problem for the Navier-Stokes equations 
on a complete $n$-dimensional Riemannian manifold $M$ 
\begin{align}  \label{eq:NS} 
&\partial_t v + \nabla_{\displaystyle v} v - Lv = - \mathrm{grad}\, p, 
\quad 
\mathrm{div}\, v = 0 
\\  \label{eq:ic} 
&v(0, x)=v_0(x). 
\end{align} 
The symbol $\nabla$ denotes the covariant derivative and 
$ 
L  = \Delta - 2 r
$ 
where $\Delta$ is the Laplacian on vector fields and $r$ is the Ricci curvature of $M$. 
Dropping the linear term $Lv$ from the first equation in \eqref{eq:NS} leads to 
the Euler equations of hydrodynamics 
\begin{equation}\label{eq:Euler}
\partial_t v + \nabla_{\displaystyle v} v = - \mathrm{grad}\, p, 
\quad 
\mathrm{div}\, v = 0. 
\end{equation}

\medskip

Most of the work on well-posedness of the Navier-Stokes equations has focused on 
the cases where $M$ is either a domain in $\mathbb{R}^n$ or the flat $n$-torus $\mathbb{T}^n$. 
In fundamental contributions J. Leray and E. Hopf established existence of 
an important class of weak solutions described as those divergence-free vector fields 
$v$ in $L^\infty([0,\infty), L^2) \cap L^2([0,\infty), H^1)$ 
which solve the Navier-Stokes equations in the sense of distributions 
and satisfy 
\begin{equation} \label{eq:energy} 
\| v(t)\|_{L^2}^2 + 4\int_0^t \| \mathrm{Def}\, v(s)\|_{L^2}^2 ds \leq \|v_0\|_{L^2}^2 
\quad 
\text{and}
\quad 
\lim_{t\searrow 0}\|v(t)-v_0\|_{L^2}=0 
\end{equation} 
for any $0 \leq t < \infty$ and where $\mathrm{Def}\, v = \frac{1}{2}(\nabla v + \nabla v^\transpose)$ 
is the so-called deformation tensor. 
When $n=2$ using interpolation inequalities and energy estimates 
it is possible to show that the Leray-Hopf solutions are unique and regular 
but the problem is in general open for $n= 3$, see e.g. \cite{cf} or \cite{mb}. 

There have also been  studies on curved spaces, which with few exceptions 
 have been confined to compact manifolds (possibly with boundary), 
see e.g. \cite{ta} and the references therein. 
In a recent paper Chan and Czubak \cite{cc} studied the Navier-Stokes equation 
on the hyperbolic plane $\mathbb{H}^2$ and more general non-compact manifolds 
of negative curvature. 
In particular, using the results of Anderson \cite{an} and Sullivan \cite{su} 
on the Dirichlet problem at inifnity, they showed that in the former case 
the Cauchy problem \eqref{eq:NS}-\eqref{eq:ic} admits non-unique Leray-Hopf solutions.

\smallskip

Our goal in this note is to provide a direct formulation of the non-uniqueness 
of the Leray-Hopf solutions on $\mathbb{H}^2$ which turns out to rely on the specific form of 
the Hodge decomposition for 1-forms (or vector fields) in this case.
We also  show that no such phenomenon can occur in the hyperbolic space $\HH^n$ with $n\ge 3$. 
As a by-product, we  describe the corresponding Hamiltonian setting of the Euler equations 
on complete Riemannian manifolds (in particular, hyperbolic spaces).

We point out that this type of non-uniqueness cannot be found in the Euler equations. 
Furthermore, it is of a different nature than the examples constructed e.g., by Shnirelman \cite{sh} or 
 De Lellis and Sz\'ekelyhidi \cite{ds}. 
On the other hand, it is similar to non-uniqueness of solutions of the Navier-Stokes equations 
defined in unbounded domains of the higher-dimensional Euclidean space, cf. Heywood \cite{he}. 
%as we discuss below. 

%%%%%%%%%%%%%%%%%%%%%%
%%%%%%%%%%%%%%%%%%%%%%%%%%%
\bigskip

\section{Stationary harmonic solutions of the Euler equations} 

Our main result is summarized in the following theorem. 

\begin{theorem} \label{thm:main}
\mbox{}
\begin{itemize} 
\item[(i)] 
There exists an infinite-dimensional space of stationary $L^2$ harmonic solutions of 
the Euler equations on $\mathbb{H}^2$.
\item[(ii)] 
There are no stationary $L^2$ harmonic solutions of the Euler equations 
on $\mathbb{H}^n$ for any $n > 2$. 
\end{itemize} 
\end{theorem} 

\begin{proof} 
Recall the Hamiltonian formulation of the Euler equations \eqref{eq:Euler} 
on a complete Riemannian manifold $M$, see e.g. \cite{AK}. 
Consider the Lie algebra $\frak{g}_{\mathrm{reg}}=\mathrm{Vect}_\mu(M)$
of (sufficiently smooth) divergence-free vector fields on $M$ with finite $L^2$ norm. 
Its dual space $\frak{g}^\ast_{\mathrm{reg}}$ has a natural description as the quotient space 
$\Omega^1_{L^2}/\overline{ d\Omega^0_{L^2} }$ 
of the $L^2$ 1-forms modulo (the $L^2$ closure of) the exact 1-forms on $M$. 
Namely, the pairing between cosets $[\beta]\in \Omega^1_{L^2}/\overline{ d\Omega^0_{L^2} }$ of 1-forms $\beta \in \Omega^1_{L^2}$ and vector fields $w\in \mathrm{Vect}_\mu(M)$ is given by 
$$
\langle [\beta],w\rangle:=\int_M (\iota_w \beta)\, d\mu\,,
$$
where $\iota_w$ is the contraction of a differential form with a vector field $w$, and $\mu$ is the Riemannian volume form on $M$.

Let $A: \frak{g}_{\mathrm{reg}} \to \frak{g}^\ast_{\mathrm{reg}}$ denote the inertia operator 
defined by the Riemannian metric. 
The operator $A$ assigns to a vector field $v \in \mathrm{Vect}_\mu(M)$ 
the coset $[v^\flat]$ of the corresponding 1-form $v^\flat$  via the pairing given by the  metric. The coset is defined as the 1-form up to adding differentials of the $L^2$ functions on $M$. 
Thus, in the Hamiltonian framework the Euler equation reads 
$$
\frac{d}{dt} [v^\flat] = - L_{\displaystyle v} [v^\flat],
$$
where $[v^\flat] \in \Omega^1_{L^2}/\overline{ d\Omega^0_{L^2} }$
and $L_{\displaystyle v}$ is the Lie derivative in the direction of the vector field $v$. 

The space $\Omega^1_{L^2}$ of the $L^2$ 1-forms on a complete manifold $M$ 
admits the Hodge-Kodaira decomposition 
$$
\Omega^1_{L^2} 
= 
\overline{ d\Omega^0_{L^2} }\oplus\overline{ \delta\Omega^2_{L^2} }\oplus \mathcal{H}^1_{L^2},
$$ 
where the first two summands denote the $L^2$ closures of the images 
of the operators $d$ and $\delta$, while $\mathcal{H}^1_{L^2}$ is the space of 
the $L^2$ harmonic 1-forms on $M$. 
Therefore, we have a natural representation of the dual space 
$$
\frak{g}^\ast_{\mathrm{reg}}=\overline{\delta\Omega^2_{L^2}}\oplus \mathcal{H}^1_{L^2}.
$$

It turns out that the summand of the harmonic forms in the above representation corresponds to 
 steady solutions of the Euler equation. Namely, one has the following proposition. 

\begin{proposition} \label{prop} 
Each harmonic 1-form on a complete manifold $M$ which belongs to $L^2\cap L^4$ 
defines a steady solution of the Euler equation \eqref{eq:Euler} on $M$.
\end{proposition}

%If $M$ is of infinite volume we cannot bound the $L^2$-norm by the $L^4$-norm.

\begin{proof}[Proof of Proposition \ref{prop}] 
Let $\alpha$ be a bounded $L^2$ harmonic 1-form on $M$. 
Let $v_\alpha$ denote the divergence-free vector field corresponding to $\alpha$, 
i.e., $v_\alpha^\flat=\alpha$. 
Since the 1-form $\alpha$ is harmonic, using Cartan's formula gives 
$$
\frac{d}{dt} \alpha = -L_{\displaystyle v_\alpha}\alpha 
= 
-\iota_{\displaystyle v_\alpha}d \alpha  -d\iota_{\displaystyle v_\alpha}\alpha 
= 
- d\iota_{\displaystyle v_\alpha}\alpha.
$$
%where $\iota_{\displaystyle v_\alpha}$ is the inner product with the vector field $v_{\alpha}$. 

We claim that $\iota_{\displaystyle v_\alpha}\alpha\in \Omega^0_{L^2}$ and consequently 
$d\iota_{\displaystyle v_\alpha}\alpha\in d \Omega^0_{L^2}$. 
Indeed, by the definition of the vector field $v_\alpha$ we have
$$
\|\iota_{\displaystyle v_\alpha}\alpha\|^2_{L^2} 
= 
\int_M \big( \alpha(v_{\alpha}) \big)^2 d\mu 
= 
\|\alpha\|^4_{L^4},
$$ 
which is finite by assumption.
%\le \|\alpha\|^2_{L^\infty}\|\alpha\|^2_{L^2}
It follows that the 1-form $d\iota_{\displaystyle v_\alpha}\alpha$ must  correspond to 
the zero coset in the quotient space 
$\frak{g}^\ast_{\mathrm{reg}}= \Omega^1_{L^2}/\overline{ d\Omega^0_{L^2} }$, 
which in turn implies that $\frac{d}{dt} \alpha=0\in \frak{g}^\ast_{\mathrm{reg}}$. 
%Indeed, by definition of the field $v_\alpha$ we have
%$$
%\|i_{v_\alpha}\alpha\|^2_{L^2}=\|\alpha\|^4_{L^4}\le \|\alpha\|^2_{L^\infty}\|\alpha\|^2_{L^2}<\infty,
%$$
%since the 1-form $\alpha$ is bounded and in $L^2$. 
The latter means that the 1-form $\alpha$ defines a steady solution of the Euler equation, which proves the proposition.
\end{proof} 

%\medskip 

If $M$ is compact then the space of harmonic 1-forms 
is always finite-dimensional (and isomorphic to the deRham cohomology group $H^1(M)$). 
%It is zero for $M=\mathbb{R}^n$. 
According to a well-known result of Dodziuk \cite{do}, 
the hyperbolic space $\mathbb{H}^n$ carries no $L^2$ harmonic $k$-forms except for $k=n/2$, 
in which case it is infinite-dimensional. 
Therefore, there can be 
%Since we are interested in $1$-forms, there are 
no $L^2$ harmonic stationary solutions of the Euler equations 
on $\mathbb{H}^n$ for any $n>2$, 
which proves  part (ii) of the theorem.

\smallskip

To prove part (i) we note that for $n=2$  the space of harmonic 1-forms 
on $\mathbb{H}^2$ is infinite-dimensional. Moreover, it allows for the following construction.
Consider the subspace $\mathcal{S} \subset \mathcal{H}^1_{L^2}$ 
of 1-forms which are differentials of bounded harmonic functions whose differentials 
are in $L^2$ 
$$
\mathcal{S}  = \left\{ d\Phi~|~\Phi ~\text{is harmonic on~}  \HH^2 ~\text{and}~ d\Phi \in L^2 \right\}.
$$ 
It turns out that the subspace $\mathcal{S}$ is  already infinite-dimensional.
Indeed, let us consider the Poincar\'e model of  $\mathbb{H}^2$, i.e., 
the unit disk   $\mathbb{D}$ with the hyperbolic metric $\langle ~,~ \rangle_{h}$, which we denote by $\mathbb{D}_{h}$. 
It is conformally equivalent to the standard unit disk   with the Euclidean metric $\langle ~,~ \rangle_{e}$, denoted by $\mathbb{D}_e$. 
Bounded harmonic functions on $\mathbb{D}_{h}$ can be obtained by 
solving the Dirichlet problem on $\mathbb{D}_{e}$, i.e., 
by constructing  harmonic functions $\Phi$ on $\D$ with boundary values $\varphi$ 
prescribed on $\partial \D$.
First, the 1-form $d\Phi$ is clearly harmonic: 
$$ 
\Delta d\Phi=d\delta d\Phi=d\Delta \Phi=0. 
$$ 
Secondly, observe that
\begin{align*}  
\|d\Phi\|^2_{L^2(\D_{h})}
&=
\int_\D \langle d\Phi, d\Phi \rangle_{h} \,d\mu_{h}
= 
\int_\D \det (g^{ij}) \langle d\Phi, d\Phi \rangle_{e} \det (g_{ij}) \,d\mu_{e} 
\\
&=
\int_\D  \langle d\Phi, d\Phi \rangle_{e}  \,d\mu_{e}
= 
\|d\Phi\|^2_{L^2(\D_{e})},
\end{align*}
and
\begin{align*}  
\|d\Phi\|^4_{L^4(\D_{h})}
&=
\int_\D \langle d\Phi, d\Phi \rangle^4_{h} \,d\mu_{h}
=
\int_\D {\det}^2 (g^{ij}) \langle d\Phi, d\Phi \rangle^2_{e} \det (g_{ij}) \,d\mu_{e}
\\
&=
\int_\D (1-|z|^2)^2 \langle d\Phi, d\Phi \rangle^2_{e}  \,d\mu_{e}(z)
\le 
\int_\D  \langle d\Phi, d\Phi \rangle^2_{e}  \,d\mu_{e}
=
\|d\Phi\|^4_{L^4(\D_{e})},
\end{align*}
where $\det (g_{ij})=1/(1-|z|^2)^2$ is the determinant of the hyperbolic metric.

Furthermore, for sufficiently smooth boundary values $\varphi\in C^{1+\alpha}(\partial \D)$ 
there is a uniform upper bound for its harmonic extension inside the disk: 
$$
|d\Phi(x)| \leq C \|\varphi\|_{C^{1+\alpha}(\partial \D)}
$$ 
for any $x\in \D$  and $0 < \alpha < 1$, and some positive constant $C$, see e.g. \cite{gt}.
This implies that 
(for sufficiently smooth $\varphi$) 
the 1-forms $d\Phi$ define an infinite-dimensional subspace $\mathcal{S}$ of harmonic forms 
in $L^2\cap L^4$, which satisfy assumptions of  the proposition above. 
It follows that they define an infinite-dimensional space of 
stationary solutions of the Euler equations on the hyperbolic plane $\HH^2$.
This completes the proof of %item (i) of 
Theorem~\ref{thm:main}.
\end{proof} 

\bigskip

\section{Non-unique Leray-Hopf solutions of the Navier-Stokes equations} 

%In a recent paper \cite{cc} the authors observed that 
%a steady solution of the Euler equations defines can be related to 
%a solution of the Navier-Stokes equation by a time-dependent rescaling 
%and, moreover, that there are multiple ways %to choose the time function 
%in which this can be done so as to satisfy the constraints of the Hopf--Leray solutions.

Using the fact that suitably rescaled %(by a time-dependent function) 
steady solutions of the Euler equations also solve the Navier-Stokes system 
the authors in \cite{cc} obtained a type of ill-posedness result 
for the Leray-Hopf solutions in the hyperbolic setting. 

\begin{theorem}[\cite{cc}] \label{thmCC} 
Given a vector field $v_e = (d\Phi)^\sharp$ on $\mathbb{H}^2$ 
there exist infinitely many real-valued functions $f(t)$ for which $v_{ns} = f(t) v_e$ is a weak solution of 
the Navier-Stokes equations with decreasing energy (i.e., satisfying the Leray-Hopf conditions). 
\end{theorem} 

An immediate consequence of this result and Theorem \ref{thm:main} is  the following 

\begin{corollary} There exist infinitely many weak Leray-Hopf
solutions to the Navier-Stokes equation on $\mathbb{H}^2$. 
There are no   non-unique Leray-Hopf harmonic 
solutions to the  Navier-Stokes equation on $\HH^n$ with $n\ge 3$ arising  from the above construction. %of Theorem \ref{thmCC}.
\end{corollary}

%%%%%%%%%%%%%%%%%%%%%
\begin{rem} 
{\rm 
The phenomenon of nonuniqueness of solutions to the Navier-Stokes equation in unbounded domains  
$\Omega\subset\R^n, \,n\ge 3$, of higher-dimensional Euclidean spaces  is of similar nature, see  \cite{he}. 
Indeed, that construction is based on the existence of a harmonic function with gradient 
in $L^2$  and appropriate boundary conditions in such domains. 
%In turn, e.g., the 
The Green function 
$\Phi(x)=G(a,x)$ centered at a point $a$ outside of $\Omega$ has the decay like 
$G(a,x)\sim |x|^{2-n}$ as $x\to\infty$, so that $|d\Phi(x)|\sim |x|^{1-n}$ and hence 
$|d\Phi(x)|^2\sim |x|^{2-2n}$. Thus, for $n\ge 3$ the 1-forms $d\Phi$ belong to $L^2\cap L^4$ 
on $\Omega$. The corresponding divergence-free vector fields $(d\Phi)^\sharp$ provide examples of
stationary Eulerian solutions in $\Omega$ (with nontrivial boundary conditions) and can be used to construct 
time-dependent weak  solutions $v_{ns}=f(t)(d\Phi)^\sharp$ to the Navier-Stokes equation in $\Omega$, 
as in Theorem \ref{thmCC}. 
}
\end{rem}
%%%%%%%%%%%%%%%%%%%%%%%%%% 

\bigskip

%%%%%
\section{Appendix} 

To make this note self-contained we provide here some details of the construction 
of the weak solutions given in \cite{cc}. 
It will be convenient to rewrite the Navier-Stokes equations \eqref{eq:NS} 
in the language of differential forms 
\begin{equation} \label{eq:NSdf} 
\partial_t v^\flat + \nabla_{\displaystyle v} v^\flat -  \Delta v^\flat + 2r(v^\flat) = -dp, 
\quad 
\delta v^\flat =0 
\end{equation} 
where $\delta v^\flat = -\mathrm{div}\, v$ 
and 
$\Delta v^\flat = d\delta v^\flat + \delta d v^\flat$ 
is the Laplace-deRham operator on 1-forms. 

Let $v$ be the vector field $v_{ns}=f(t)(d\Phi)^\sharp$ on $\mathbb{H}^2$ as in Theorem \ref{thmCC}. 
Since the 1-form $d\Phi$ is harmonic one only needs to compute the covariant derivative term and 
the Ricci term: 
$$ 
\nabla_{\displaystyle v_{ns}} v_{ns}^\flat = \frac{1}{2} f^2(t) \, d|d\Phi|^2
\qquad 
\text{and} 
\qquad 
2r(v_{ns}^\flat) = - 2f(t) d\Phi.
$$ 
Direct computation, taking into account the fact that for $\mathbb{H}^2$ we have $r=-1$, 
shows that
both terms can be absorbed by the pressure term, 
so that the pair $(v_{ns}^\flat, p)$,  
where 
$p := (2f(t)-f'(t))\Phi - 1/2 f^2(t) |d\Phi|^2$ 
satisfies the equations \eqref{eq:NSdf}. 

Finally, a quick inspection shows that any differentiable function $f(t)$ satisfying 
$$ 
f^2(t) + 4 \int_0^t f^2(s) \, ds \leq f^2(0)
$$ 
yields a vector field $v_{ns}$ which satisfies the remaining conditions in \eqref{eq:energy} 
required of a Leray-Hopf solution.

%%%%%%%%%%%%%%%%%%%%%%%%%%%%%%%%%
%%%%%%%%%%%%%%%%%%%%%%%%%%%%%%%%%%%

\end{document}